\documentclass{amsart}
\usepackage{all2015}
\newcommand{\Alt}{\mathrm{Alt}}
\newcommand{\comment}[1]{}

\newcommand{\rip}{(\cdot|\cdot)}

\newcommand{\lbr}{[\cdot,\cdot]}
\newcommand{\bi}{\mathbf{i}}

\begin{document}

\title[Orthogonal and unitary tensor decomposition]{Orthogonal and unitary tensor decomposition from an
algebraic perspective}

\author{Ada Boralevi}
\address{
Department of Mathematics and Computer Science\\
Technische Universiteit Eindhoven\\
P.O. Box 513, 5600 MB Eindhoven, The Netherlands}
\email{a.boralevi@tue.nl}
\author{Jan Draisma}
\address{
Department of Mathematics and Computer Science\\
Technische Universiteit Eindhoven\\
P.O. Box 513, 5600 MB Eindhoven, The Netherlands;\\
and VU Amsterdam, The Netherlands}
\email{j.draisma@tue.nl}
\author{Emil Horobe\c{t}}
\address{
Department of Mathematics and Computer Science\\
Technische Universiteit Eindhoven\\
P.O. Box 513, 5600 MB Eindhoven, The Netherlands}
\email{e.horobet@tue.nl}
\author{Elina Robeva}
\address{Department of Mathematics\\ 
University of California Berkeley\\ 
775 Evans Hall, Berkeley, CA 94720, USA}
\email{erobeva@berkeley.edu}
\maketitle

\begin{abstract}
While every matrix admits a singular value decomposition, in which the
terms are pairwise orthogonal in a strong sense, higher-order tensors
typically do not admit such an orthogonal decomposition. Those that
do have attracted attention from theoretical computer science and
scientific compu\-ting. We complement this existing body of literature with
an algebro-geometric analysis of the set of orthogonally decomposable
tensors.

More specifically, we prove that they
form a real-algebraic variety defined by polynomials of degree at most
four. The exact degrees, and the corresponding polynomials, are different
in each of three times two scenarios: ordinary, symmetric, or alternating
tensors; and real-orthogonal versus complex-unitary.  A key feature of
our approach is a surprising connection between orthogonally decomposable
tensors and semisimple algebras---associative in the ordinary
and symmetric settings and of compact Lie type in the alternating setting.
\end{abstract}

\tableofcontents

\vspace{-2cm}
\begin{figure}
\begin{center}
\includegraphics[width=.9\textwidth]{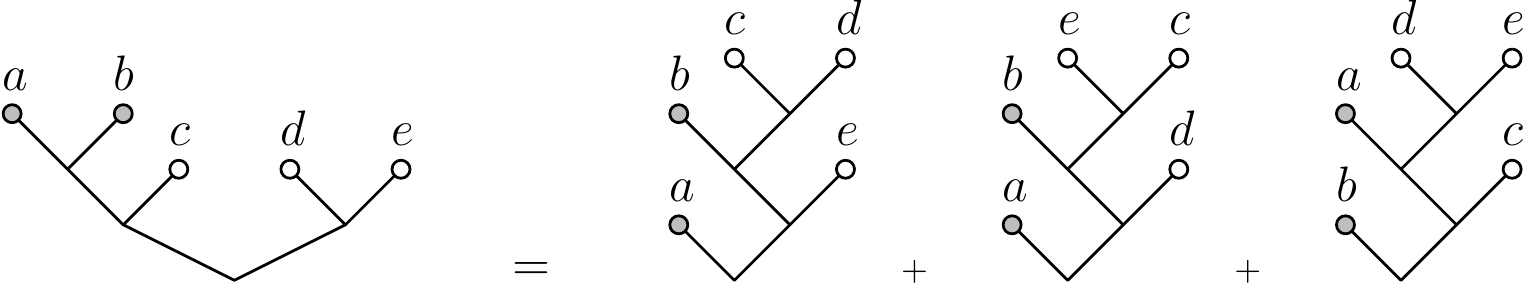}
\caption{Quartic equations for alternatingly udeco tensors; see
Lemma~\ref{lm:AltUdecoCross}.}
\label{fig:cross}
\end{center}
\end{figure}

\clearpage

\section{Introduction and results}

By the singular value decomposition, any complex $m \times n$-matrix $A$
can be written as $A=\sum_{i=1}^k u_i v_i^T,$ where $u_1,\ldots,u_k \in
\CC^m$ and $v_1,\ldots,v_k \in \CC^n$ are sets of nonzero, pairwise
orthogonal vectors with respect to the standard Hermitian forms on
these spaces. The singular values $||u_i|| \cdot ||v_i||,$ including their
multiplicities, are uniquely determined by $A,$ and if these are all
distinct, then so are the terms $u_i v_i^T.$

If $m=n$ and $A$ is symmetric, then the $u_i$ and $v_i$ can be chosen
equal. And if, on the other hand, $A$ is skew-symmetric, then $k$
is necessarily even, say $k=2\ell,$ and one can choose
$v_i=u_{\ell+i}$ for
$i=1,\ldots,\ell$ and $v_i=-u_{i-\ell}$ for $i=\ell+1,\ldots,n,$ so that the
terms can be grouped into pairs of the form $u_i v_i^T - v_i u_i^T$ for
$i=1,\ldots,\ell.$ Note that the two-dimensional spaces $\langle u_i,v_i
\rangle_\CC$ for $i=1,\ldots,\ell$ are pairwise perpendicular.

In this paper we consider {\em higher-order} tensors in a tensor product
$V_1 \otimes \cdots \otimes V_d$ of finite-dimensional vector spaces $V_i$
over $K \in \{\RR, \CC\},$ where the tensor product is also
over $K.$ We assume that each $V_i$ is equipped with a positive-definite
inner product $\rip,$ Hermitian if $K=\CC.$

\begin{de}\label{de:Odeco}
A tensor in $V_1 \otimes \cdots \otimes V_d$ is called
{\em orthogonally decomposable} ({\em odeco}, if $K=\RR$) or {\em unitarily decomposable}
({\em udeco}, if $K=\CC$) if it can be written as
\[
\sum_{i=1}^k v_{i1} \otimes \cdots \otimes v_{id},
\]
where for each $j$ the vectors $v_{1j},\ldots,v_{kj}$ are nonzero and
pairwise orthogonal in $V_j.$
\end{de}

We use the adverb {\em unitarily} for $K=\CC$ to stress that we
have fixed Hermitian inner products rather than symmetric bilinear
forms. Note that orthogonality implies that the number $k$ of terms
is at most the minimum of the dimensions of the $V_i,$ so odeco tensors
form a rather low-dimensional subset of the space of all tensors; see
Proposition~\ref{prop:Dimensions}.

Next we consider tensor powers of a single, finite-dimensional $K$-space
$V.$ We write $\Sym_d(V)$ for the subspace of $V^{\otimes d}$ consisting
of all {\em symmetric tensors}, i.e., those fixed by all permutations
of the tensor factors.

\begin{de}\label{de:SymOdeco}
A tensor in $\Sym_d(V)$ is called {\em symmetrically odeco}
(if $K=\RR$) or {\em symmetrically udeco} (if $K=\CC$) if it can be written as
\[
\sum_{i=1}^k \pm v_i^{\otimes d}
\]
where the vectors $v_1,\ldots,v_k$ are nonzero, pairwise orthogonal
vectors in $V.$
\end{de}

The signs are only required when $K=\RR$ and $d$ is even, as they
can otherwise be absorbed into the $v_i$ by taking a $d$-th root of
$-1.$ Clearly, a symmetrically odeco or udeco tensor is symmetric
and odeco or udeco in the earlier sense. The converse also holds; see
Proposition~\ref{prop:SymAsOrdinary}.

Our third scenario concerns the space $\Alt_d(V) \subseteq V^{\otimes d}$
consisting of all {\em alternating tensors}, i.e., those $T$ for which
$\pi T=\sgn(\pi)T$ for each permutation $\pi$ of $[d].$ The simplest
alternating tensors are the alternating product tensors
\[ v_1 \wedge \cdots \wedge v_d:=\sum_{\pi \in S_d} \sgn(\pi)
v_{\pi(1)} \otimes \cdots \otimes v_{\pi(d)}. \]
This tensor is nonzero if and only if $v_1,\ldots,v_d$ form an
independent set, and it changes only by a scalar factor upon replacing these
vectors by another basis of the space $\langle v_1,\ldots,v_d \rangle.$ We
say that this subspace is {\em represented} by the alternating product
tensor.

\begin{de}\label{de:AltOdeco}
A tensor in $\Alt_d(V)$ is called {\em alternatingly
odeco} or {\em alternatingly udeco}
if it can be written as
\[
\sum_{i=1}^k v_{i1} \wedge \cdots \wedge v_{id},
\]
where the $k\cdot d$ vectors $v_{11},\ldots,v_{kd}$ are nonzero and
pairwise orthogonal.
\end{de}

Equivalently, this means that the tensor is a sum of $k$ alternating
product tensors that represent pairwise orthogonal $d$-dimensional
subspaces of $V;$ by choosing orthogonal bases in each of these spaces
one obtains a decomposition as above. In particular, $k$ is at most
$\lfloor n/d \rfloor.$ For $d \geq 3,$ alternatingly odeco tensors are
not odeco in the ordinary sense unless they are zero; see
Remark~\ref{re:AlternatingNotOdeco}.

By quantifier elimination, it follows that the set of odeco or udeco
tensors is a semi-algebraic set in $V_1 \otimes \cdots \otimes V_d,$
i.e., a finite union of subsets described by polynomial equations and
(weak or strict) polynomial inequalities; here this space is considered
as a real vector space even if $K=\CC.$ A simple compactness argument (see
Proposition~\ref{prop:Closed}) also shows that they form a closed subset
in the Euclidean topology, so that only weak inequalities are needed.
However, our main result says that, in fact, only {\em equations} are
needed, and that the same holds in the symmetrically or alternatingly
odeco or udeco regimes.

\begin{thm}[Main Theorem]
For each integer $d \geq 3,$ for $K \in \{\RR,\CC\},$ and for all
finite-dimensional inner product spaces $V_1,\ldots,V_d$ and $V$ over
$K,$ the odeco/udeco tensors in $V_1 \otimes \cdots \otimes V_d,$ the
symmetrically odeco/udeco tensors in $\Sym_d(V),$ and the alternatingly
odeco/udeco tensors in $\Alt_d(V),$ form real algebraic varieties defined
by polynomials of degree given in the following table.
\end{thm}

\begin{table}[h]
\begin{tabular}{|l||r|r|r|}
\hline
Degrees of equations & odeco (over $\RR$) & udeco (over $\CC$)\\
\hline
\hline
symmetric & $2$ (associativity) & $3$ (semi-associativity)\\
\hline
ordinary & $2$ (partial associativity) & $3$ (partial semi-associativity)\\
\hline
alternating & $2$ (Jacobi) and $4$ (cross) & $3$ (Casimir) and $4$ (cross)\\
\hline
\end{tabular}
\end{table}

\begin{re}
Several remarks are in order:
\begin{enumerate}
\item Unlike for $d=2,$ for $d \geq 3$ the decomposition in Definitions
\ref{de:Odeco}, \ref{de:SymOdeco}, and \ref{de:AltOdeco} is always unique
in the sense that the terms are uniquely determined, regardless of
whether some of their norms coincide; see Proposition~\ref{prop:Unique}.

\item A direct consequence of the fact that we work with Hermitian forms,
is that even when $K=\CC$ the varieties above are {\em real algebraic}
only, except in the following three degenerate cases: $\dim V$ or some
$\dim V_i$ equals zero; $\dim V_1=1$ for some $i$, so that the set
of odeco/udeco tensors equals the affine cone over the Segre product
of $\PP V_1,\ldots,\PP V_d$; or $d>\dim V/2$ in the alternating case,
so that the set of alternatingly odeco/udeco tensors equals the affine
cone over the Grassmannian of $k$-subspaces of $V$. So apart from these
cases, we need to allow polynomial equations in which both coordinates
with respect to a $\CC$-basis appear and their complex conjugates.

\item We will describe the polynomials defining these varieties in detail
later on, but here is a high-level perspective.  In the odeco case,
the equations of degree two guarantee that some algebra associated
to a tensor is associative (in the ordinary and symmetric cases) or Lie
(in the alternating case), and the equations of degree four come
from a certain polynomial identity satisfied by the cross product
on $\RR^3$.  These degree-four equations are not always required:
e.g., for $\Alt_3(V)$ they can be discarded (leaving only our degree-two
equations) if and only if the real vector space $V$ has dimension $\leq 7$
(see Remark~\ref{re:CrossNotNeeded}).

\item The udeco case is more involved: the equations of degree three
express that some algebra with a bi-semilinear product is (partially)
semi-associative in a sense to be defined below, or, in the case of
alternatingly udeco tensors, that a variant of the Casimir operator
commutes with the multiplication.

\item The listed degrees are minimal in the sense that there are no
linear equations in the odeco case and no quadratic equations in the
udeco case---again, except in degenerate cases. Moreover, the equations
of degree four for the alternating case cannot be simply discarded.
But we do not know whether, instead of the degree-four equations,
lower-degree equations might also suffice.

\item More generally, we do not know whether the equations that we give
generate the prime ideal of all polynomial equations vanishing on our
real algebraic varieties.
\end{enumerate}
\end{re}

The remainder of this paper is organised as follows. In
Section~\ref{sec:Background} we discuss some background and earlier
literature; in particular we show that our degree-two equations follow
from those obtained by the fourth author in \cite{robeva_odeco}, so that
our Main Theorem implies \cite[Conjecture 3.2]{robeva_odeco} from that
paper at a set-theoretic level and over the real numbers.

In Section~\ref{sec:OrderThree} we prove the Main Theorem for tensors
of order three. We first treat odeco tensors, and then the more
involved case of udeco tensors.  The proofs for symmetrically odeco
and udeco three-tensors are the simplest, and those for ordinary odeco
and udeco three-tensors build upon them. The alternatingly odeco and
udeco three-tensors require a completely separate treatment. Then, in
Section~\ref{sec:HigherOrder} we derive the theorem for higher-order
tensors for ordinary, symmetric, and alternating tensors consecutively.
We conclude in Section~\ref{sec:Concluding} with some open questions.

\section{Background} \label{sec:Background}

In this section we collect background results on orthogonally decomposable
tensors, and connect our paper to earlier work on them.

\begin{prop} \label{prop:Closed}
The set of (ordinary, symmetrically, or alternatingly) odeco or udeco
tensors is closed in the Euclidean topology.
\end{prop}

\begin{proof}
We give the argument for symmetrically udeco tensors; the same works in
the other cases. Thus consider the space $V=\CC^n$ with the standard inner
product, let $\lieg{U}_n$ be the unitary group of that inner product,
and consider the map
\[ \varphi: \lieg{U}_n \times \PP V \to \PP \Sym_d(V),\
((u_1|\ldots|u_n),[\lambda_1:\ldots:\lambda_n])
	\mapsto \left[\sum_{i=1}^n \lambda_i u_i^{\otimes d}\right].  \]
Here $\PP$ stands for projective space and where $u_i$ is the $i$-th
column of the unitary matrix $u.$ The key point is that this map
is well-defined and continuous, since the expression between the last
square brackets is never zero by linear independence of the
$u_i^{\otimes d}.$
Now $\varphi$ is a continuous map whose source is a compact topological
space, hence $\im \varphi$ is a closed subset of $\PP \Sym_d(V).$ But then
the pre-image of $\im \varphi$ in $\Sym_d(V) \setminus \{0\}$ is also
closed, and so is the union of this pre-image with $\{0\}.$ This is
the set of symmetrically udeco tensors in $\Sym_d(V).$
\end{proof}

\begin{prop} \label{prop:Unique}
For $d \geq 3,$ any (ordinary, symmetrically, or alternatingly) odeco
or udeco tensor has a {\em unique} orthogonal decomposition.
\end{prop}

In the ordinary case this was proved in \cite[Theorem 3.2]{Zhang01}.

\begin{proof}
We give the argument for ordinary odeco tensors. Consider an orthogonal
decomposition
\[ T=\sum_{i=1}^k v_{i1} \otimes \cdots \otimes v_{id} \]
of an odeco tensor $T \in V_1 \otimes \cdots \otimes V_d.$ Contracting
$T$ with an arbitrary tensor $S \in V_3 \otimes \cdots \otimes V_d$
via the inner products on $V_3,\ldots,V_d$ leads to a tensor
\[ T'=\sum_{i=1}^k \lambda_i v_{i1} \otimes v_{i2} \]
where $\lambda_i$ is the inner product of $S$ with $v_{i3} \otimes
\cdots \otimes v_{id}.$ Now the above is a singular value decomposition
for the two-tensor $T',$ of which, for $S$ sufficiently general, the
singular values $|\lambda_i| \cdot ||v_{i1}|| \cdot ||v_{i2}||$ are
all distinct. Thus $v_{11},\ldots,v_{k1}$ are, up to nonzero scalars,
uniqely determined as the singular vectors (corresponding to the nonzero
singular values) of the pairing of $T$ with a
sufficiently general $S.$ And these vectors determine the corresponding
terms, since the $i$-th term equals $v_{i1}$ tensor the pairing of $T$
with $v_{i1},$ divided by $||v_{i1}||^2.$

The arguments in the symmetric or alternating case, as well as in the
udeco case, are almost identical. We stress that, as permuting the
first two factors commutes with contracting the last $d-2$ factors,
the contraction of a symmetric or alternating tensor is a symmetric
or alternating matrix. Also, in the alternating case, rather than
contracting with a general $S$, we contract with a general alternating
product tensor $S=u_1 \wedge \cdots \wedge u_{d-2}$. This has the effect
of intersecting the space spanned by $v_{i1},\ldots,v_{id}$ with the
orthogonal complement of the space spanned by $u_1,\ldots,u_{d-2}$.
\end{proof}

Note that the proof of this proposition yields a simple randomised
algorithm for deciding whether a tensor is odeco or udeco, and for finding
a decomposition when it exists. At the heart of this algorithm is the
computation of an ordinary singular-value decomposition for a small
matrix. For much more on algorithmic issues see
\cite{Batselier15,Kolda15,Salmi09,Zhang01}.

The uniqueness of the orthogonal decomposition makes it easy to compute
the dimensions of the real-algebraic varieties in our Main Theorem.

\begin{prop} \label{prop:Dimensions}
Let $n:=\dim_K V$,  $l:=\lfloor \frac{n}{d}
\rfloor$, and assume that the dimensions $n_i:=\dim_K V_i$ are
in increasing order  $n_1 \leq \ldots \leq n_d$.
 Then, the dimensions of the real-algebraic varieties of odeco/udeco
 tensors are given in the following table.
\end{prop}

\begin{table}[h]
\begin{tabular}{|l||r|r|r|}
\hline
Dimension over $\RR$ & odeco & udeco\\
\hline
\hline
symmetric & $n+\binom{n}{2}$ & $2n+n(n-1)$ \\
\hline
ordinary & $n_1+\sum_{j=1}^d \frac{n_1(2n_j-n_1-1)}{2}$ &
$2n_1+\sum_{j=1}^d n_1(2n_j-n_1-1)$\\
\hline
alternating & $l+\frac{ld(2n-(l+1)d)}{2}$ & $2l+ld(2n-(l+1)d)$ \\
\hline
\end{tabular}
\end{table}

\begin{proof}
In the symmetric case, a symmetrically odeco/udeco tensor encodes $n$ pairwise
perpendicular points in $\PP V$. For the first point we have $n-1$ degrees
of freedom over $K$. The second point is chosen from the projective
space orthogonal to the first point, so this yields $n-2$ degrees of freedom,
etc. Summing up, we obtain $\binom{n}{2}$ degrees of freedom over $K$ for
the points. In addition, we have $n$ scalars from $K$ for the individual
terms. If $K=\CC$ we multiply by two to obtain the real dimension. Since each
odeco/udeco tensor has a unique decomposition, the dimension of the odeco/udeco
variety is the same as the dimension of the space of $n$ pairwise orthogonal points and
$n$ scalars.

The computation for the ordinary case is the same, except that only $n_1$
pairwise perpendicular projective points are chosen from each $V_j$.

In the alternating case, an alternatingly odeco/udeco tensor encodes $l$
pairwise perpendicular $d$-dimensional $K$-subspaces of $V$. The first
space is an arbitrary point on the $d(n-d)$-dimensional Grassmannian of
$d$-subspaces, the second an arbitrary point on the $d(n-2d)$-dimensional
Grassmannian of $d$-subspaces in the orthogonal complement of the first,
etc. Add $l$ degrees of freedom for the scalars, and
if $K=\CC$, multiply by $2$. By uniqueness of the decomposition, the dimension
of the odeco/udeco variety is preserved and as given in the table above.
\end{proof}

\bigskip

Over the last two decades, orthogonal tensor decomposition has been
studied intensively from a scientific computing perspective (see,
e.g., \cite{Comon94,K01,K03,Chen09,Kolda15}), though the alternating
case has not received much attention so far. The paper \cite{Chen09}
gives a characterisation of orthogonally decomposable tensors in terms
of their {\em higher-order SVD} \cite{Lathauwer00b}, which is different
from the real-algebraic characterisation in our Main Theorem.  One of
the interesting properties of an orthogonal tensor decomposition with
$k$ terms is that discarding the $r$ terms with smallest norm yields the
best rank-$r$ approximation to the tensor; see \cite{Vannieuwenhoven14},
where it is also proved that in general, tensors are not {\em optimally
truncatable} in this manner.

In general, tensor decomposition is NP-hard \cite{HL}. The decomposition
of odeco tensors, however, can  be found efficiently.  The vectors in
the decomposition of an odeco tensor are exactly the attraction points of
the {\em tensor power method} and are called {\em robust} eigenvectors.
Because of their efficient decomposition, odeco tensors have been used
in machine learning, in particular for learning latent variables in
statistical models \cite{AGHKT}.  More recent work in this direction
concerns overcomplete latent variable models \cite{Anandkumar14}.

In \cite{robeva_odeco}, the fourth author describes all eigenvectors of
symmetrically odeco tensors in terms of the robust ones, and conjectures
the equations defining the variety of symmetrically odeco tensors.
Formulated for the case of ordinary tensors instead, this conjecture is as
follows. Let $V_1,\ldots,V_d$ be real inner product spaces, and consider
an odeco tensor $T \in V_1 \otimes \cdots \otimes V_d$ with orthogonal
decomposition $ T=\sum_{i=1}^k v_{i1} \otimes \cdots \otimes v_{id}.$
Now take two copies of $T$, and contract these in their $l$-th components
via the inner product $V_l \times V_l \to \RR$. By orthogonality of
the $v_{il},\ i=1,\ldots,k$, after regrouping the tensor factors, the
resulting tensor is
\[ \sum_{i=1}^k \left( ||v_{il}||^2 \bigotimes_{j \neq l} (v_{ij}
\otimes v_{ij}) \right) \in \bigotimes_{j \neq l}
(V_j \otimes V_j); \]
we write $T *_l T$ for this tensor. It is clear from this expression that
$T *_l T$ is multi-symmetric in the sense that it lies in the subspace
$\bigotimes_{j \neq l} \Sym_2(V_j)$. The fourth author conjectured that
this (or rather, its analogue in the symmetric setting) characterises
odeco tensors. This is now a theorem, which follows from the proof of
our main theorem (see Remark~\ref{re:Elina}).

\begin{thm} \label{thm:Elina}
$T \in V_1 \otimes \cdots \otimes V_d$ is odeco if and only
if for all $l=1,\ldots,d$ we have
\[ T *_l T \in \bigotimes_{j \neq l} \Sym_2(V_j). \]
\end{thm}

This concludes the discussion of background to our results. We now
proceed to prove the main theorem in the case of order-three tensors.


\section{Tensors of order three} \label{sec:OrderThree}

In all our proofs below, we will encounter a finite-dimensional vector
space $A$ over $K=\RR$ or $\CC$ equipped with a positive-definite
inner product $(\cdot|\cdot),$ as well as a bi-additive product $A \times A
\to A,\ (x,y) \mapsto x \cdot y$ which is bilinear if $K=\RR$ and
bi-semilinear if $K=\CC.$ The product will be either commutative or
anti-commutative. Moreover, the inner product will be {\em compatible}
with the product in the sense that $(x\cdot y|z)=(z \cdot x|y).$ An {\em
ideal} in $(A,\cdot)$ is a $K$-subspace $I$ such that $I \cdot A \subseteq
I$---by (anti-)commutativity we then also have $A \cdot I \subseteq
I$---and $A$ is called {\em simple} if $A \neq \{0\}$ and $A$ contains
no nonzero proper ideals. We have the following well-known result.

\begin{lm} \label{lm:Orthoplement}
The orthogonal complement $I^\perp$ of any ideal $I$ in $A$ is an ideal,
as well. Consequently, $A$ splits as a direct sum of pairwise orthogonal
simple ideals.
\end{lm}

\begin{proof}
We have $(A \cdot I^\perp | I) = (I \cdot A |I^\perp)=\{0\}.$ The
second statement follows by induction on $\dim A.$
\end{proof}

\subsection{Symmetrically odeco three-tensors} \label{sec:S3Vodeco}

In this subsection, we fix a finite-dimensional real inner product space
$V$ and characterise odeco tensors in $\Sym_3(V).$ We have $\Sym_3(V)
\subseteq V^{\otimes 3} \cong (V^*)^{\otimes 2} \otimes V,$ where the
isomorphism comes from the linear isomorphism $V \to V^*,\ v \mapsto
(v|\cdot).$ Thus a general tensor $T \in \Sym_3(V)$ gives rise to a bilinear
map $V \times V \to V,\ (u,v) \mapsto u\cdot v,$ which has the following
properties:
\begin{enumerate}
\item $u\cdot v=v\cdot u$ for all $u,v \in V$ (commutativity, which
follows from the fact that $T$ is invariant under permuting the first
two factors); and
\item $(u\cdot v|w)=(u\cdot w|v)$ (compatibility with the inner product,
which follows from the fact that $T$ is invariant under permuting the
last two factors).
\end{enumerate}
Thus $T$ gives $V$ the structure of an $\RR$-algebra equipped with a
compatible inner product. The following lemma describes the quadratic
equations from the Main Theorem.

\begin{lm} \label{lm:SymOdecoAss}
If $T$ is symmetrically odeco, then $(V,\cdot)$ is associative.
\end{lm}

\begin{proof}
Write $T=\sum_{i=1}^k v_i^{\otimes 3}$ where $v_1,\ldots,v_k$ are pairwise
orthogonal nonzero vectors. Then we find, for $x,y,z \in V,$ that
\[ x\cdot (y\cdot z)=x\cdot \left(\sum_i (v_i|y)(v_i|z) v_i\right)=\sum_i
(v_i|x)(v_i|y)(v_i|z)(v_i|v_i)=(x\cdot y)\cdot z, \]
where we have used that $(v_i|v_j)=0$ for $i \neq j$ in the second
equality.
\end{proof}

\begin{prop} \label{prop:AssSymOdeco}
Conversely, if $(V,\cdot)$ is associative, then $T$
is symmetrically odeco.
\end{prop}

\begin{proof}
By Lemma~\ref{lm:Orthoplement}, $V$ has an orthogonal decomposition
$V=\bigoplus_i U_i$ where the subspaces $U_i$ are (nonzero)
simple ideals.  Correspondingly, $T$ decomposes as an element of
$\bigoplus_i \Sym_3(U_i).$ Thus it suffices to prove that each $U_i$ is
one-dimensional. This is certainly the case when the multiplication $U_i
\times U_i \to U_i$ is zero, because then any one-dimensional subspace
of $U_i$ is an ideal in $V,$ hence equal to $U_i$ by simplicity. If the
multiplication map is nonzero, then pick an element $x \in
U_i$ such that the
multiplication $M_x: U_i \to U_i,\ y \mapsto x \cdot y$ is nonzero. Then
$\ker M_x$ is an ideal in $V,$ because for $z \in V$ we have
\[ x\cdot(\ker M_x \cdot z)=(x \cdot \ker M_x)\cdot z=\{0\}, \]
where we use associativity. By simplicity of $U_i,$ $\ker
M_x=\{0\}.$  Now define a new bilinear multiplication $*$ on $U_i$ via
$y*z:=M_x^{-1}(y\cdot z).$ This multiplication is commutative, has $x$
as a unit element, and we claim that it is also associative. Indeed,
\[ ((x \cdot y) * z) * (x \cdot v)
=
M_x^{-1} (M_x^{-1} ((x \cdot y) \cdot z) \cdot (x \cdot v))
=
y \cdot z \cdot v
= (x \cdot y) * (z * (x \cdot v)),
\]
where we used associativity and commutativity of $\cdot$ in the
second equality. Since any element is a multiple of $x,$ this proves
associativity. Moreover, $(U_i,*)$ is simple; indeed, if $I$
is ideal, then $M_x^{-1} (U_i \cdot I) \subseteq I$ and
hence
\[ U_i \cdot (x \cdot I)=(U_i \cdot x) \cdot I=U_i \cdot I
\subseteq x\cdot I, \]
so that $x \cdot I$ is an ideal in $(U_i,\cdot);$ and therefore
$I=\{0\}$ or $I=U_i$.

Now $(U_i,*)$ is a simple, associative $\RR$-algebra with
$1,$ hence isomorphic to a matrix algebra over a division ring. As it
is also commutative, it is isomorphic to either $\RR$ or $\CC.$ If it
were isomorphic to $\CC,$ then it would contain a square root of $-1,$
i.e., an element $y$ with $y*y=-x,$ so that $y \cdot y=-x \cdot x.$
But then
\[ 0<(x\cdot y|x\cdot y)=(y \cdot y|x \cdot x)=-(x \cdot x|x \cdot
x)<0,
\]
a contradiction. We conclude that $U_i$ is one-dimensional, as
desired.
\end{proof}

Lemma~\ref{lm:SymOdecoAss} and Proposition~\ref{prop:AssSymOdeco} imply the Main
Theorem for symmetrically odeco three-tensors, because the identity
$x \cdot (y \cdot z)=(x \cdot y) \cdot z$ expressing associativity
translates into quadratic equations for the tensor $T.$

\subsection{Ordinary odeco three-tensors}

\begin{figure}
\begin{center}
\includegraphics{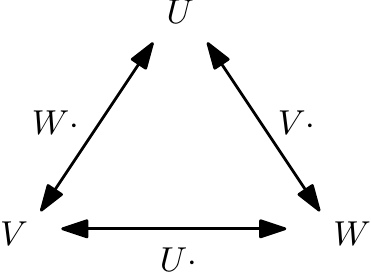}
\caption{$U \cdot (V+W)=W+V,$ and similarly with $U,V,W$ permuted.}
\label{fig:UVW}
\end{center}
\end{figure}

In this subsection, we consider a general tensor $T$ in a tensor product
$U \otimes V \otimes W$ of real, finite-dimensional inner product
spaces. Via the inner products, $T$ gives rise to a bilinear map $U
\times V \to W,$ and similarly with the three spaces
permuted. Consider the external direct sum $A:=U
\oplus V \oplus W$ of $U,V,W$, and equip $A$
with the inner product $\rip$ that restricts to the given inner products
on $U,V,W$ and that makes these spaces pairwise perpendicular.

Taking cue from the symmetric case, we construct a bilinear
product $\cdot: A \times A \to A$ as follows: the product in $A$ of two
elements in $U,$ or two elements in $V,$ or in $W,$ is defined as zero;
$\cdot$ restricted to $U \times V$ is the map into $W$ given by $T;$
etc.---see Figure~\ref{fig:UVW}.  The tensor in $\Sym_3(A)$ describing the
multiplication is the symmetric embedding of $T$ from \cite{Ragnarsson13}.

As in the symmetrically odeco case, the algebra has two fundamental properties:
\begin{enumerate}
\item it is commutative: $x\cdot y=y \cdot x$ by definition; and
\item the inner product is compatible: $(x\cdot y|z)=(x\cdot z|y).$ For
instance, if $x \in U,y \in V,z \in W,$ then both sides equal the inner
product of the tensor $x \otimes y \otimes z$ with $T;$ and if $y,z \in
W,$ then both sides are zero both for $x \in U$ (so that $x\cdot y,x\cdot z \in V,$
which is perpendicular to $W$) and for $x \in W$ (so that $x\cdot y=x\cdot z=0$)
and for $x \in V$ (so that $x\cdot y,x\cdot z \in U \perp W$).
\end{enumerate}

We are now interested in {\em homogeneous} ideals $I \subseteq A$ only, i.e.,
ideals such that $I=(I \cap U) \oplus (I \cap V) \oplus (I \cap W).$
We call $A$ \emph{simple} if it is nonzero and does not contain proper, nonzero
homogeneous ideals. We will call an element of $A$ homogeneous if it
belongs to one of $U,V,W.$ Next, we derive a polynomial identity for
odeco tensors.

\begin{lm} \label{lm:OdecoAss}
If $T$ is odeco, then for all homogeneous $x,y,z$ where $x$ and $z$ belong
to the same space ($U,V,$ or $W$), we have $(x\cdot y)\cdot z=x\cdot (y
\cdot z).$
\end{lm}

We will refer to this property as {\em partial associativity}.

\begin{proof}
If $x,y,z$ all belong to the same space, then both products are zero.
Otherwise, by symmetry, it suffices to check the case where $x,z \in
U$ and $y \in V.$ Let $T=\sum_i u_i \otimes v_i \otimes w_i$ be an
orthogonal decomposition of $T.$ Then we have
\[
(x\cdot y) \cdot z =
\left(\sum_i (u_i|x)(v_i|y) w_i\right) \cdot z =
\sum_i (u_i|x)(v_i|y)(w_i|w_i)(z|u_i)=
x \cdot (y \cdot z),
\]
where we have used that $(w_i|w_j)=0$ for $i \neq j$ in the second
equality.
\end{proof}

\begin{prop} \label{prop:AssOdeco}
Conversely, if $(A,\cdot)$ is partially associative, then $T$ is odeco.
\end{prop}


\begin{proof}
By a version of Lemma~\ref{lm:Orthoplement} restricted to homogeneous
ideals, $A$ is the direct sum of pairwise orthogonal, simple homogeneous
ideals $I_i$. Accordingly, $T$ lies in $\bigoplus_i (I_i \cap U) \otimes
(I_i \cap V) \otimes (I_i \cap W)$. Thus it suffices to prove that $T$
is odeco under the additional assumption that $A$ itself is simple
and that $\cdot$ is not identically zero.

By symmetry, we may assume that $V \cdot (U+W) \neq \{0\}$.  For $u \in
U$, let $M_u:V+W \to W+V$ be multiplication with $u$. By commutativity and
partial associativity, the $M_u$, for $u\in U$, all commute. By compatibility
of $\rip$, each $M_u$ is symmetric with respect to the inner
product on $V+W$, and hence orthogonally diagonalisable. 

Consequently,
$V+W$ splits as a direct sum of pairwise orthogonal simultaneous eigenspaces
\[ (V+W)_\lambda:=\{v+w \in V+W \mid u\cdot(v+w)=\lambda(u)(w+v)
\text{ for all } u\in U\},\]
where $\lambda$ runs over $U^*$. Suppose we are given $v+w \in
(V+W)_\lambda$ and $v'+w' \in (V+W)_\mu$ with $\lambda \neq \mu$.
Then $v+w$ and $v'+w'$ are perpendicular and for each $u \in V$ we have
\[ (u|(v+w) \cdot (v'+w'))=(u\cdot(v+w)|v'+w')=
\lambda(u)(v+w|v'+w')=0, \]
hence $(v+w)\cdot(v'+w')=0$. We conclude that for each $\lambda$ the space
\[ (V+W)_\lambda \oplus \left[(V+W)_\lambda \cdot (V+W)_\lambda\right] \]
is a homogeneous ideal in $A$. By simplicity and the fact that $M_u
\neq 0$ for at least some $u$, $A$ is equal to this ideal for some
nonzero $\lambda$. Pick an $x \in U$ such that $\lambda(x)=1$, so that
$x\cdot (v+w)=w+v$ for all $v \in V,\ w \in W$. In particular, for $v,v'
\in V$ we have $(M_x v|M_x v')=(M_x^2 v|v')=(v|v')$, so that the restrictions
$M_x:V \to W$ and $M_x:W \to V$ are mutually inverse isometries.

By the same construction, we find an element $z \in W$ such that $z \cdot
(u+v)=v+u$ for all $u \in U,\ v\in V$. Let $T'$ be the image of $T$ under
the linear map $M_x \otimes I_V \otimes M_z: U \otimes V \otimes W \to V
\otimes V \otimes V$. We claim that $T'$ is symmetrically odeco. Indeed,
let $*:V\times V \to V$ denote the bilinear map associated to $T'$. We
verify the conditions from Section~\ref{sec:S3Vodeco}. First,
\[ v*v'=
(x\cdot v)\cdot(z\cdot v')=
z \cdot ((x \cdot v) \cdot v')=
z \cdot ((v' \cdot x)\cdot v)=
(z \cdot v) \cdot (v' \cdot x)=
v'*v,\]
where we have repeatedly used commutativity and partial associativity
(e.g., in the second
equality, to the elements $x\cdot v,z$ belonging to the same space
$W$).
Second, we have
\[
(v*v'|v'')=
((x \cdot v) \cdot (z \cdot v')|v'')=
((x \cdot v) | v' \cdot (z \cdot v''))=
(v | (x \cdot v') \cdot (z \cdot v''))=
(v | v'*v'').
\]
Hence $T'$ is, indeed, and element of $\Sym_3(V)$. Finally, we have
\begin{align*}
(v*v')*v''&=
(x \cdot ((x \cdot v) \cdot (z \cdot v'))) \cdot (z \cdot v'')=
x \cdot ((z \cdot v'') \cdot ((x \cdot v) \cdot (z \cdot v')))\\
&=x \cdot (((z \cdot v'') \cdot (x \cdot v)) \cdot (z \cdot v'))=
x \cdot ((v*v'') \cdot (z \cdot v'))=
(v*v'')*v',
\end{align*}
which, together with commutativity, implies associativity of $*$.
Hence $T'$ is (symmetrically) odeco by Proposition~\ref{prop:AssOdeco},
and hence so is its image $T$ under the tensor product $M_x \otimes I_V
\otimes M_z$ of linear isometries.
\end{proof}

\begin{re} \label{re:Elina}
The condition that $(x\cdot y)\cdot z= x\cdot (y\cdot z)$ for, say,
$x,z \in W$ and $y \in V$ translates into the condition that the
contraction $T *_1 T \in (V \otimes V) \otimes (W \otimes W)$ lies in
$\Sym_2(V) \otimes \Sym_2(W)$. Thus Proposition~\ref{prop:AssOdeco} implies
Theorem~\ref{thm:Elina} in the case of three factors. The case of more
factors follows from the case of three factors and flattening as in
Proposition~\ref{prop:Flatten}.
\end{re}

\subsection{Alternatingly odeco three-tensors}

In this subsection we consider a tensor $T \in \Alt_3(V),$ where $V$ is a
finite-dimensional real vector space with inner product $\rip.$ Via
$\Alt_3(V) \subseteq V^{\otimes 3} \cong (V^*)^{\otimes 2} \otimes V$
such a tensor gives rise to a bilinear map $V \times V \to V,\ (u,v)
\mapsto [u,v],$ which gives $V$ the structure of an algebra. Now,

\begin{enumerate}
\item as the permutation $(1,2)$ maps $T$ to $-T,$ we have $[u,v]=-[v,u];$ and
\item as $(2,3)$ does the same, we have $([u,v]|w)=-([u,w]|v)=([w,u]|v),$ so that the inner product is compatible with the product.
\end{enumerate}

The following lemma gives the degree-two equations from the
Main Theorem.

\begin{lm} \label{lm:AltOdecoJacobi}
If $T$ is alternatingly odeco, then $\lbr$ satisfies the Jacobi identity.
\end{lm}

\begin{proof}
Let $T=\sum_{i=1}^k u_i \wedge v_i \wedge w_i$ be an orthogonal
decomposition of $T,$ and set $V_i:=\langle u_i,v_i,w_i \rangle.$ Then
$V$ splits as the direct sum of $k$ ideals $V_i$ and one further ideal
$V_0:=(\bigoplus_{i=1}^k V_i)^\perp.$ The restriction of the
bracket to $V_0$ is zero, so it suffices to verify the Jacobi identity on each
$V_i.$ By scaling the bracket, which preserves both the Jacobi
identity and the set of alternatingly odeco tensors, we
achieve that $u_i,v_i,w_i$ can be taken of norm one.
Then we have
\[ [u_i,v_i]=w_i,\ [v_i,w_i]=u_i, \text{ and } [w_i,u_i]= v_i, \]
which we recognise as the multiplication table of $\RR^3$ with the cross
product $\times,$ isomorphic to the Lie algebra
$\liea{so}_3(\RR).$
\end{proof}

The following lemma gives the degree-four equations from the
Main Theorem.

\begin{lm} \label{lm:AltOdecoCasimir}
If $T$ is alternatingly odeco, then
$[x,[[x,y],[x,z]]]=0$ for all $x,y,z \in V$.
\end{lm}

We will refer to this identity as the {\em first cross product identity}.

\begin{proof}
By the proof of Lemma~\ref{lm:AltOdecoJacobi}, if $T$ is odeco, then $V$
splits as an orthogonal direct sum of ideals $V_1,\ldots,V_k$ that are
isomorphic, as Lie algebras with compatible inner products, to scaled
copies of $\RR^3$ with the cross product, and possibly an additional
ideal $V_0$ on which the multiplication is trivial. Thus it suffices to
prove that the lemma holds for $\RR^3$ with the cross product. But there
it is immediate: if $[[x,y],[x,z]]$ is nonzero, then the two arguments
span the plane orthogonal to $x$, hence their cross product is a scalar
multiple of $x$.
\end{proof}

We now prove the Main Theorem for alternatingly odeco three-tensors.

\begin{prop} \label{prop:JacobiCasimirAltOdeco}
Conversely, if the bracket $\lbr$ on $V$ satisfies the Jacobi identity
and the first cross product identity, then $T$ is alternatingly
odeco.
\end{prop}

\begin{proof}
By Lemma~\ref{lm:Orthoplement} the space $V$ splits into pairwise
orthogonal, simple ideals $V_i.$ Correspondingly, $T$ lies in $\bigoplus_i
\Alt_3 V_i,$ where the sum is over those $V_i$ where the
bracket is nonzero. These are
simple real Lie algebras equipped with a compatible inner product, hence
compact Lie algebras. Let $\liea{g}$ be one of these, so $\liea{g}$
satisfies the first cross product identity. Then so does the complex
Lie algebra $\liea{g}_\CC:= \CC \otimes \liea{g},$ which is
semisimple. For $\liea{g} \cong
\liea{so}_3(\RR),$ we have $\liea{g}_\CC \cong \liea{sl}_2(\CC),$
i.e., the Dynkin diagram of $\liea{g}_\CC$ has a single node. The
classification of simple compact Lie algebras (see, e.g.,
\cite{knapp}) shows that, if $\liea{g}$
is not isomorphic to $\liea{so}_3(\RR),$ then the Dynkin diagram of
$\liea{g}_\CC$ contains at least one edge, so that $\liea{g}_\CC$
contains a copy of $\liea{sl}_3(\CC).$ However,
this $8$-dimensional complex Lie algebra does {\em not} satisfy the
cross product identity, as for instance
\[
[E_{11}-E_{33},[[E_{11}-E_{33},E_{12}],[E_{11}-E_{33},E_{23}]]]
=
2 E_{13} \neq 0,\]
where $E_{ij}$ is the matrix with zeroes everywhere except for a $1$
on position $(i,j)$. Hence $\liea{g} \cong \liea{so}_3(\RR)$ is
three-dimensional, and $T$ is alternatingly odeco.
\end{proof}

\begin{re} \label{re:CrossNotNeeded}
The classification of simple compact Lie algebras shows that, after
$\liea{so}_3(\RR)$, the next smallest one is $\liea{su}_3(\CC)$, of
dimension $8$---recall that  $\liea{so}_4(\RR)$, of dimension $6$,
is a direct sum of two copies of $\liea{so}_3$, arising from left and
right multiplication of quaternions by norm-one quaternions. Thus our
degree-four equations from the cross product identity are not necessary
for $\dim V<8$.
\end{re}

\subsection{Symmetrically udeco three-tensors}

In the complex udeco setting, all proofs towards Main Theorem are more
complicated than in the real case. The reason for this is that the
bi-additive map $(u,v) \mapsto u \cdot v$ associated to a tensor is no
longer bi-linear. Rather, it is bi-semilinear, i.e., it satisfies $(cu)
\cdot (dv)=\overline{cd} (u \cdot v)$ for complex coefficients $c,d.$ To
appreciate why this causes trouble, consider the usual associativity
identity:
\[ (x \cdot y) \cdot z = x \cdot (y \cdot z). \]
If the product is bi-semilinear, then the left-hand side depends linearly
on both $x$ and $y$ and semilinearly on $z,$ while the right-hand
side depends linearly on $y$ and $z$ and semilinearly on $x.$  Hence,
except in trivial cases, one should not expect associativity to hold
for bi-semilinear products. This explains the need for more complex
polynomial identities (pun intended).

\vskip.03in

In this subsection, $V$ is a complex, finite-dimensional vector space
equipped with a positive-definite Hermitian inner product $\rip,$
and $T$ is an element of $\Sym_3(V).$ There is a canonical linear
isomorphism $V \to V^s,\ v \mapsto (v|\cdot),$ where $V^s$ is the space of
semilinear functions $V \to \CC.$ Through $\Sym_3(V) \subseteq V^{\otimes
3} \cong (V^s)^{\otimes 2} \otimes V,$ the tensor $T$ gives rise to
a bi-semilinear product $V \times V \to V,\ (u,v) \mapsto u \cdot
v.$ Moreover:
\begin{enumerate}
\item since $T$ is invariant under permuting the first two factors,
$\cdot$ is commutative; and
\item since $T$ it is invariant under permuting the
last two factors, we find that $(u\cdot v|w)=(u \cdot w|v).$ Note that,
in this identity, both sides are semilinear in all three vectors $u,v,w.$
\end{enumerate}

The following lemma gives the degree-three equations of the Main Theorem.

\begin{lm} \label{lm:SymUdecoAss}
If $T$ is symmetrically udeco, then for all $x,y,z,u \in V$ we have
\[ x \cdot (y \cdot (z \cdot u)) = z \cdot (y \cdot (x \cdot u))
\text{ and } (x \cdot y) \cdot (z \cdot u)=(x \cdot u) \cdot
(z \cdot y).\]
\end{lm}

We call a commutative operation $\cdot$ satisfying the identities in the
lemma {\em semi-associative}.  It is clear that any commutative and associative
operation is also semi-associative, but the converse does not hold. Note
that, since the product is bi-semilinear, both sides of the first identity
depend semilinearly on $x,z,u$ but linearly on $y,$ while both parts of
the second identity depend linearly on all of $x,y,z,u.$

\begin{proof}
Let $T=\sum_i v_i^{\otimes 3}$ be an orthogonal decomposition of $T.$
Then we have
\[ z \cdot u = \sum_i (v_i|z)(v_i|u) v_i \]
and
\[ y \cdot (z \cdot u) = \sum_i (v_i|y)(z|v_i)(u|v_i)(v_i|v_i) v_i \]
by the orthogonality of the $v_i.$ We stress that the coefficient
$(v_i|z)(v_i|u)$ has been transformed into its complex conjugate
$(z|v_i)(u|v_i).$ Next, we find
\[ x \cdot (y \cdot (z \cdot u))=\sum_i
(v_i|x)(y|v_i)(v_i|z)(v_i|u)(v_i|v_i)(v_i|v_i)v_i \]
and this expression is invariant under permuting $x,z,u$ in
any manner. This proves the first identity.

For the second identity, we compute
\[ (x \cdot y) \cdot (z \cdot u) = \sum_i (v_i|v_i)^2
(x|v_i)(y|v_i)(z|v_i)(u|v_i) v_i, \]
which is clearly invariant under permuting $x,y,z,u$ in any
manner.
\end{proof}

\begin{prop} \label{prop:AssSymUdeco}
Conversely, if $\cdot$ is semi-associative, then $T$ is symmetrically
udeco.
\end{prop}

In fact, in the proof we will only use the first identity.  The second
identity will be used later on, for the case of ordinary udeco
three-tensors.

\begin{ex}
To see how the identities for semi-associativity, in
Lemma~\ref{lm:SymUdecoAss}, transform into equations for  symmetrically
udeco tensors we consider $2\times 2\times 2$ tensors. Let
$\{e_1,e_2\}$ be an orthonormal basis of $\mathbb{C}^2$ and we represent
a general element of $\Sym_3(\mathbb{C}^2)$  by
\begin{align*}T&=
t_{3,0} e_1\otimes e_1\otimes e_1 +
t_{2,1} (e_1\otimes e_1\otimes e_2 + e_1 \otimes e_2 \otimes e_1
+ e_2 \otimes e_1 \otimes e_1)  \\
&+t_{1,2} (e_1\otimes e_2\otimes e_2 + e_2 \otimes e_1 \otimes e_2 +
e_2 \otimes e_2 \otimes e_1) +
t_{0,3} e_2\otimes e_2\otimes e_2.
\end{align*}
Then the identities for semi-associativity are translated into two complex equations. If we separate the real and imaginary parts of these two complex equations then we get that the real algebraic variety of $2\times 2\times 2$ symmetrically udeco tensors is given by the following four real equations (note that they are invariant under conjugation): 
\begin{align*}
f_1 =& -t_{1,2}^2 \overline{t}_{1,2} + t_{0,3} t_{2,1} \overline{t}_{1,2} - t_{1,2} \overline{t}_{1,2}^2 - t_{1,2} t_{2,1} \overline{t}_{2,1} +
       t_{0,3} t_{3,0} \overline{t}_{2,1} + t_{1,2} \overline{t}_{0,3} \overline{t}_{2,1} -\\& t_{2,1} \overline{t}_{1,2} \overline{t}_{2,1} - t_{3,0} \overline{t}_{2,1}^2 -
       t_{2,1}^2 \overline{t}_{3,0} + t_{1,2} t_{3,0} \overline{t}_{3,0} + t_{2,1} \overline{t}_{0,3} \overline{t}_{3,0} + t_{3,0} \overline{t}_{1,2} \overline{t}_{3,0};\\
f_2 =& -t_{1,2}^2 \overline{t}_{1,2} + t_{0,3} t_{2,1} \overline{t}_{1,2} + t_{1,2} \overline{t}_{1,2}^2 - t_{1,2} t_{2,1} \overline{t}_{2,1} +
       t_{0,3} t_{3,0} \overline{t}_{2,1} - t_{1,2} \overline{t}_{0,3} \overline{t}_{2,1}+\\&  t_{2,1} \overline{t}_{1,2} \overline{t}_{2,1} + t_{3,0} \overline{t}_{2,1}^2 -
       t_{2,1}^2 \overline{t}_{3,0} + t_{1,2} t_{3,0} \overline{t}_{3,0} - t_{2,1} \overline{t}_{0,3} \overline{t}_{3,0} - t_{3,0} \overline{t}_{1,2} \overline{t}_{3,0};\\
f_3 =& -t_{1,2}^2 \overline{t}_{0,3} + t_{0,3} t_{2,1} \overline{t}_{0,3} - t_{1,2} t_{2,1} \overline{t}_{1,2} + t_{0,3} t_{3,0} \overline{t}_{1,2} -
       t_{0,3} \overline{t}_{1,2}^2 - t_{2,1}^2 \overline{t}_{2,1} +\\& t_{1,2} t_{3,0} \overline{t}_{2,1} + t_{0,3} \overline{t}_{0,3} \overline{t}_{2,1} -
       t_{1,2} \overline{t}_{1,2} \overline{t}_{2,1} - t_{2,1} \overline{t}_{2,1}^2 + t_{1,2} \overline{t}_{0,3} \overline{t}_{3,0} + t_{2,1} \overline{t}_{1,2} \overline{t}_{3,0};\\
f_4 =& -t_{1,2}^2 \overline{t}_{0,3} + t_{0,3} t_{2,1} \overline{t}_{0,3} - t_{1,2} t_{2,1} \overline{t}_{1,2} + t_{0,3} t_{3,0} \overline{t}_{1,2} +
       t_{0,3} \overline{t}_{1,2}^2 - t_{2,1}^2 \overline{t}_{2,1} +\\& t_{1,2} t_{3,0} \overline{t}_{2,1} - t_{0,3} \overline{t}_{0,3} \overline{t}_{2,1} +
       t_{1,2} \overline{t}_{1,2} \overline{t}_{2,1} + t_{2,1} \overline{t}_{2,1}^2 - t_{1,2} \overline{t}_{0,3} \overline{t}_{3,0} - t_{2,1} \overline{t}_{1,2} \overline{t}_{3,0}.
\end{align*}
The polynomials $f_1,f_2,f_3$ and $f_4$ generate a real codimension $2$ variety,
as expected by Proposition~\ref{prop:Dimensions}. We do not know whether
the ideal generated by these polynomials is prime or not.
\end{ex}

\begin{proof}[Proof of Proposition~\ref{prop:AssSymUdeco}.]
By Lemma~\ref{lm:Orthoplement}, $V$ is the direct sum of pairwise
orthogonal, simple ideals $V_i.$ Correspondingly, $T$ lies in
$\bigoplus_i \Sym_3(V_i).$ We want to show that those ideals on which
the multiplication is nonzero are one-dimensional. Thus we may assume
that $V$ itself is simple with nonzero product.

Then the elements $x \in V$ for which the semilinear map $M_x:V \to V,\
y \mapsto x \cdot y$ is identically zero form a proper ideal in $V,$
which is zero by simplicity. Hence for any nonzero $x \in V$ the map $M_x$
is nonzero.

Now consider, for nonzero $x \in V,$ the space $W:=\ker M_x.$ We claim
that $W$ is a proper ideal. First, $W$ also equals $\ker
M_x^2,$ because if $M_x^2 v=0,$ then $(M_x^2 v|v)=(x(xv)|v)=(xv|xv)=0,$
so $x v=0.$ We have
\[ M_{x}^2(V \cdot W)=x \cdot (x \cdot (V \cdot W))=V \cdot
(x \cdot (x \cdot W))=\{0\}. \]
Here we used semi-associativity in the second equality.
So $V \cdot W \subseteq \ker M_x^2=W,$ as claimed. Hence $W$ is zero.

Fixing any nonzero $x \in V,$ we define a new operation on $V$ by
\[ y*z:=M_x^{-1}(y \cdot z). \]
Since $M_x^{-1}$ is semilinear, $*$ is bilinear, commutative, and has
$x$ as a unit element. We claim that it is also associative.
For this we need to prove that
\[ v \cdot M_x^{-1} (z \cdot y)=z \cdot M_x^{-1} (v \cdot y)
\]
holds for all $y,z,v \in V.$ Write $y=M_x^2 y',$ so that
$x\cdot(x \cdot(z \cdot y'))=z \cdot y$ and $x \cdot (x \cdot (v \cdot
y'))=v \cdot y$ by semi-associativity. Then the equation to
be proved reads
\[ v \cdot (x \cdot (z \cdot y')) = z \cdot (x \cdot (v \cdot y')), \]
which is another instance of semi-associativity.

Furthermore, any nonzero element $y \in V$ is invertible in $(V,*)$
with inverse $M_y^{-1}(x \cdot x).$ We conclude that $(V,*,+)$ is a
finite-dimensional field extension of $\CC,$ hence equal to $\CC.$
\end{proof}

\subsection{Ordinary udeco three-tensors}

In this subsection, $U,V,W$ are three finite-dimensional complex vector spaces
equipped with Hermitian inner products $\rip,$ and $T$ is a tensor in $U
\otimes V \otimes W.$ Then $T$ gives rise to bi-semilinear maps $U \times
V \to W,$ $V \times U \to W,$ etc. Like for ordinary three-tensors in
the real case, we equip $A:=U \oplus V \oplus W$ with the bi-semilinear
product $\cdot$ arising from these maps and with the inner product which
restricts to the given inner products on $U,V,$ and $W,$ and is zero
on all other pairs. By construction:
\begin{enumerate}
	\item $(A,\cdot)$ is commutative, and
	\item the inner product is compatible.
\end{enumerate}

The following lemma gives the degree-three equations from the Main
Theorem.

\begin{lm} \label{lm:UdecoAss}
If $T$ is udeco, then
\begin{enumerate}
\item for all $u,u',u'' \in U$ and $v \in V$ we have $u \cdot (u' \cdot
(u'' \cdot v))=u'' \cdot (u' \cdot (u \cdot v));$
\item for all $u \in U,$ $v,v' \in V,$ and $w \in W$ we have
$u \cdot (v \cdot (w \cdot v'))=w \cdot (v \cdot (u \cdot
v'))$ and $(u \cdot v) \cdot (w \cdot v')=(u \cdot v') \cdot
(w \cdot v);$
\end{enumerate}
and the same relations hold with $U,V,W$ permuted in any manner.
\end{lm}

We call $\cdot$ \emph{partially semi-associative} if it satisfies these
conditions.

\begin{proof}
Let $T=\sum_i u_i \otimes v_i \otimes w_i$ be an orthogonal decomposition
of $T.$ Then we have
\begin{align*}
u'' \cdot v &= \sum_i (u_i | u'')(v_i | v) w_i, \\
u' \cdot (u'' \cdot v) &= \sum_i
(u_i|u')(u''|u_i)(v|v_i)(w_i|w_i) v_i, \text{ and}\\
u \cdot (u' \cdot (u'' \cdot v)) &= \sum_i
	(u_i|u)(u'|u_i)(u_i|u'')(v_i|v)(w_i|w_i)(v_i|v_i)w_i,
\end{align*}
which is invariant under swapping $u$ and $u''.$ The second
identity is similar. For the last identity, we have
\[ (u \cdot v) \cdot (w \cdot v')=\sum_i
(u|u_i)(v|v_i)(w_i|w_i)(w|w_i)(v'|v_i)(u_i|u_i), \]
which is invariant under swapping $v$ and $v'.$
\end{proof}

The following proposition implies the Main Theorem for
three-tensors over $\CC.$

\begin{prop} \label{prop:AssUdeco}
Conversely, if $\cdot$ is partially semi-associative, then $T$ is
udeco.
\end{prop}

\begin{proof}
By a version of Lemma~\ref{lm:Orthoplement} for homogeneous ideals $I
\subseteq A,$ i.e., those for which $I=(I \cap U) \oplus (I \cap V)
\oplus (I \cap W),$ $A$ splits as a direct sum of nonzero, pairwise
orthogonal, homogeneous ideals $I_i$ that each do not contain proper,
nonzero homogeneous ideals, and $T$ lies in $\bigoplus_i (I_i \cap U)
\otimes (I_i \cap V) \otimes (I_i \cap W),$ where the sum is over those
$i$ on which the multiplication $\cdot$ is nontrivial. Thus we may assume
that $A$ itself is nonzero, contains no proper nonzero ideals, and has
nontrivial multiplication. We then need to prove that each of $U,V,W$
is one-dimensional.

Without loss of generality, $U \cdot V$ is a non-zero subset of $W.$
The $u \in U$ for which the multiplication $M_u: V+W \to W+V,\ (v+w)
\mapsto u\cdot w + u \cdot v$ is zero form a homogeneous, proper ideal
in $A,$ which is zero by simplicity.

Pick an $x \in U,$ and let $Q:=\ker M_x \subseteq V+W,$ so that $Q
\cdot Q \subseteq U.$ We want to prove that $Q \oplus (Q \cdot Q)$
is a proper homogeneous ideal in $A.$ First, $\ker M_x$ equals $\ker M_x^2$ because
$0=(x(xv)|v)=(xv|xv)$ implies $xv=0.$ Now $U \cdot Q \subseteq Q$ because
\[ M_x^2(U \cdot Q)=x \cdot (x \cdot (U
\cdot Q))=U \cdot (x \cdot (x \cdot Q))=\{0\} \]
by partial semi-associativity.

Next, let $R$ be the orthogonal complement of $Q$ in $V+W$. We have
$(Q \cdot R | U)=(Q \cdot U|R)=\{0\},$ so that $Q \cdot R=\{0\},$ and
therefore $(V+W) \cdot Q=(Q+R)\cdot Q = Q \cdot Q$. It remains to check
whether $V \cdot (Q\cdot Q)\subseteq Q,$ and similarly for $W$. This is
true since, for $v \in Q \cap V$ and $w \in Q \cap W,$ we have
\[ x \cdot (V \cdot (w \cdot v)) = w \cdot (V \cdot (x \cdot
v))=\{0\} \]
by partial semi-associativity. We have now proved that $Q \oplus (Q
\cdot Q)$ is a proper homogeneous ideal in $A.$ Hence $Q=0$ by simplicity.

We conclude that $M_x$ is a bijection $V+W \to W+V$ for each nonzero $x
\in U.$  Similarly, $M_z$ is a bijection $U+V \to V+U$ for each nonzero
$z \in W.$ Fixing nonzero $x \in U$ and nonzero $z \in W,$ define a new
multiplication $*$ on $V$ by
\[ v * v':=(x \cdot  v) \cdot (z \cdot v') \in W \cdot U \subseteq V. \]
This operation is commutative by the third identity in partial
associativity, and it is $\CC$-linear. Moreover, for each nonzero
$v' \in V$ and each $v'' \in V$ there is an element $v \in V$ such
that $v*v'=v'',$ namely, $M_x^{-1} M_{z \cdot v'}^{-1} v'',$ which is
well-defined since also the element $z \cdot v' \in U$ is nonzero. Thus
$(V,*)$ is a commutative division algebra over $\CC,$ and by Hopf's
theorem \cite{Hopf}, $\dim_\CC V=1.$
\end{proof}

\subsection{Alternatingly udeco three-tensors}

In this section, $V$ is a finite-dimensio\-nal complex inner product space. An alternating
tensor $T \in \Alt_3(V) \subseteq V \otimes V \otimes V \cong V^s \otimes
V^s \otimes V$ gives rise to a bi-semilinear multiplication $V \times V
\to V, (a,b) \mapsto [a,b]$ that satisfies $[a,b]=-[b,a]$ and $([a,b]|
c)=-([a,c]| b).$ Just like the multiplication did not become associative
in the symmetrically udeco case, the bracket does not satisfy the Jacobi
identity in the alternatingly udeco case. However, it does satisfy the
following {\em cross product identities}.

\begin{lm} \label{lm:AltUdecoCross}
If $T$ is alternatingly udeco, then for
all $a,b,c,d,e \in V$ we have
\begin{align*}
[a,[[a,b],[a,c]]]&=0 \text{ and}\\
[[[a,b],c],[d,e]]&=[a,[[b,[c,d]],e]]+[a,[[b,[e,c]],d]]+[b,[[a,[d,e]],c]].
\end{align*}
\end{lm}

For a pictorial representation of the second identity see
Figure~\ref{fig:cross}.

\begin{proof}
In the alternatingly udeco case, the simple, nontrivial ideals of the
algebra $(V,[.,.])$ are isomorphic, via an inner product preserving
isomorphism, to $(\CC^3,c\times),$ where $\times$ is the semilinear
extension to $\CC^3$ of the cross product on $\RR^3$ and where $c$
is a scalar. Thus it suffices to prove the two identities for this
three-dimensional algebra. Moreover, both identities are homogeneous
in the sense that their validity for some $(a,b,c,d,e)$ implies
their validity when any one of the variables is scaled by a complex
number. Indeed, for the first identity this is clear, and for the second
identity this follows since all four terms are semilinear in $a,b$ and
linear in $c,d,e.$ Hence both identities follow from their validity
for the crossproduct and general $a,b,c,d,e \in \RR^3.$
\end{proof}

The cross product identities yield real degree-four equations that
vanish on the set of alternatingly odeco three-tensors.  There are
also degree-three equations, which arise as follows. Let $\mu:V\otimes
V \to V,\ (a \otimes b) \to [a,b]$ be the semilinear multiplication,
and let, conversely, $\psi:V \to V \otimes V$ be the semilinear map
determined by $(c| [a,b])=(a \otimes b|\psi(c))$---note that both sides
are linear in $a,b,c.$ Then let $H:=\mu \circ \psi: V \to V.$ Being
the composition of two semilinear maps, this is a linear map, and it
satisfies $(Ha|b)=(\psi(a)|\psi(b))=(a|Hb).$  Hence $H$ is a positive
semidefinite Hermitian map.

\begin{lm} \label{lm:AltUdecoCasimir}
If $T$ is alternatingly udeco, then $[Hx,y]=[x,Hy]$ for all $x,y \in
V.$
\end{lm}

\begin{proof}
Let $T=\sum_i u_i \wedge v_i \wedge w_i$ be an orthogonal
decomposition of $T.$ Then we have
\begin{align*}
&[Hx,y] = [\mu(\sum_i
(w_i|x) u_i \wedge v_i - (v_i|x) u_i \wedge w_i + (u_i|x) v_i \wedge
w_i),y]\\
&= \sum_i [ 2 (w_i|x)(u_i|u_i)(v_i|v_i)w_i
+ 2 (v_i|x)(u_i|u_i)(w_i|w_i)v_i
+ 2 (u_i|x)(v_i|v_i)(w_i|w_i)u_i, y]\\
&= 2 \sum_i (
(w_i|x)(u_i|u_i)(v_i|v_i)(w_i|w_i)((u_i|y)v_i-(v_i|y)u_i)\\
&\quad \quad \ + (v_i|x)(u_i|u_i)(w_i|w_i)(v_i|v_i)((w_i|y)u_i-(u_i|y)w_i)\\
&\quad \quad \ + (u_i|x)(v_i|v_i)(w_i|w_i)(u_i|u_i)((v_i|y)w_i-(w_i|y)v_i)).
\end{align*}
Now we observe that the latter expression is skew-symmetric in $x$
and $y$, so that this equals $-[Hy,x]=[x,Hy]$.
\end{proof}

\begin{re}
For a real, compact Lie algebra $\liea{g},$ the positive semidefinite
matrix $H$ constructed above is a (negative) scalar multiple of the
Casimir element in its adjoint action \cite{knapp}; this is why we call
the identity in the lemma the {\em Casimir identity}. Complexifying
$\liea{g}$ and its invariant inner product to a semilinear algebra with
an invariant Hermitian inner product, we obtain an algebra satisfying the
degree-three equations of the lemma. Hence, since for $\dim V
\geq 8$ there exist other compact Lie algebras, these equations do not
suffice to characterise alternatingly udeco three-tensors in general,
though perhaps they do so for $\dim V \leq 7$.
\end{re}

\begin{ex}
The lemma yields cubic equations satisfied by alternatingly udeco
tensors. Here is one of these, with $V=\CC^6$ and $t_{ijk}$
the coefficient of $e_i \otimes e_j \otimes e_k$:
\begin{align*}
&t_{1, 4, 5} t_{2, 3, 4} \bar{t}_{1, 3, 5} - t_{1, 3, 4} t_{2, 4, 5} \bar{t}_{1, 3, 5} + t_{1, 2, 4} t_{3, 4, 5} \bar{t}_{1, 3, 5} + t_{1, 4, 6} t_{2, 3, 4}
\bar{t}_{1, 3, 6} \\
-&t_{1, 3, 4} t_{2, 4, 6} \bar{t}_{1, 3, 6} + t_{1, 2, 4} t_{3, 4, 6}
\bar{t}_{1, 3, 6} -
t_{1, 4, 6} t_{2, 4, 5} \bar{t}_{1, 5, 6} + t_{1, 4, 5} t_{2, 4, 6}
\bar{t}_{1, 5, 6} \\
-&t_{1, 2, 4} t_{4, 5, 6} \bar{t}_{1, 5, 6} + t_{2, 4, 6} t_{3, 4, 5}
\bar{t}_{3, 5, 6} -
t_{2, 4, 5} t_{3, 4, 6} \bar{t}_{3, 5, 6} + t_{2, 3, 4} t_{4, 5, 6}
\bar{t}_{3, 5, 6} =0.
\end{align*}
This equation was first discovered as follows: working in
$V=(\ZZ/19)[i]^6$ instead of $\CC^6$ (where it is important that $19$ is $3$
modulo $4$ so that $-1$ has no square root in $\ZZ/19$), we implemented
the Cayley transform to sample general unitary matrices and from those
construct general alternatingly udeco tensors. We sampled as many as
there are degree-three monomials in the 20 variables $t_{ijk}$ plus
the 20 variables $\bar{t}_{ijk}$ (namely, $\binom{40+2}{3}=11480$), and
evaluated these monomials on the tensors. The $280$-dimensional kernel
of this matrix over $\ZZ/19$ turned out to have a basis consisting of vectors with
entries $0,1,2,17,18.$ The natural guess for lifting these equations
to characteristic zero, respectively,
yielded equations that vanish on general alternatingly udeco tensors
in characteristic zero. A similar, but smaller computation shows that
there are no degree-two equations; here the fact that these do not exist
modulo $19$ {\em proves} that they do not exist in characteristic zero.
\end{ex}

In a Lie algebra, if $[a,b]=0,$ then the left multiplications $L_a:V
\to V$ and $L_b:V \to V$ commute. This is not true in our setting,
since the Jacobi identity does not hold, but the following statement
does hold.

\begin{lm} \label{lm:PreserveKernel}
Suppose that the bracket satisfies the second cross product identity
in Lemma~\ref{lm:AltUdecoCross}, and let $a,b,c \in V$ be such that
$[a,c]=[b,c]=0.$ Then $[[a,b],c]=0.$
\end{lm}

\begin{proof}
Compute the inner product
\[ ([[a,b],c]|[[a,b],c])=-([[a,b],[[a,b],c]]|c)=([[[a,b],c],[a,b]]|c) \]
and use the identity to expand the first factor in the last inner product as
\[
[[[a,b],c],[a,b]]=[a,[[b,[c,a]],b]]+[a,[[b,[b,c]],a]]+[b,[[a,[a,b]],c]].
\]
Now each of the terms on the right-hand side is of the form $[a,x]$
or $[b,y],$ and we have $([a,x]|c)=-([a,c],x)=0$ and similarly
$([b,x]|c)=0.$ Since the inner product is positive definite, this shows
that $[[a,b],c]=0,$ as claimed.
\end{proof}

We now prove that our equations found so far suffice.

\begin{prop} \label{prop:IdsCasimirAltUdeco}
Suppose that, conversely, $T \in \Alt_3(V)$ has the properties in
Lemmas~\ref{lm:AltUdecoCross} and~\ref{lm:AltUdecoCasimir}. Then $T$
is alternatingly udeco.
\end{prop}

\begin{proof}
If $a,b \in V$ belong to distinct eigenspaces of the Hermitian
linear map $H,$ then the property that $[Ha,b]=[a,Hb]$ implies
that $[a,b]=0.$ Moreover, a fixed eigenspace of $H$ is closed under
multiplication, as for $a,b$ in the eigenspace with eigenvalue $\lambda$
and $c$ in the eigenspace with eigenvalue $\mu \neq \lambda,$ we have
\[
\overline{\lambda}([a,b]|c)=([Ha,b]|c)=-([Ha,c]|b)=-\overline{\mu}([a,c]|b)
=\overline{\mu}([a,b]|c), \]
and hence $([a,b]|c)=0.$ Thus the eigenspaces of $H$ are ideals. We may
replace $V$ by one of these, so that $H$ becomes a scalar. If the scalar is
zero, then $T$ is zero and we are done, so we assume that it is nonzero,
in which case we can scale $T$ (even by a positive real number) to
achieve that $H=1.$

Furthermore, by compatibility of the inner product and
Lemma~\ref{lm:Orthoplement}, $V$ splits further as a direct sum of
simple ideals. So to prove the proposition, in addition to $H=1,$ we
may assume that $V$ is a simple algebra and that the multiplication
is not identically zero; in this case it suffices to prove that $V$
is three-dimensional. Let $x \in V$ be a non-zero element such that
the semi-linear left multiplication $L_x:V \to V$ has minimal possible
rank. If its rank is zero, then $\langle x\rangle$ is an ideal, contrary to
the assumptions. Hence $V_1:=L_x V$ is a nonzero space, and we set
$V_0:=[V_1,V_1],$ the linear span of all products of two elements from
$V_1.$ We claim that $x \in V_0.$ For this, we note that $V_1^\perp=\ker
L_x$ and compute
\[ (\psi(x)|V_1^\perp \otimes V)=([V_1^\perp,V]|x)=([\ker
L_x,x]|V)=\{0\}. \]
Similarly, we find that $(\psi(x)|V \otimes V_1^\perp)=\{0\},$ so
$\psi(x) \in V_1 \otimes V_1$ and therefore
\[ x=Hx=\mu(\psi(x)) \in [V_1,V_1]=V_0, \]
as claimed.

Next, by the first cross product identity in
Lemma~\ref{lm:AltUdecoCross},
we find that $[x,V_0]=\{0\}.$ This implies that
$(V_0|V_1)=(V_0|[x,V])=([x,V_0]|V)=\{0\},$ so $V_0 \perp V_1.$
Furthermore, by substituting $x+s$ for $x$ in that same identity and
taking the part quadratic in $x,$ we find the identity
\[ [s,[[x,a],[x,b]]]+[x,[[s,a],[x,b]]]+[x,[[x,a],[s,b]]]=0. \]
A general element of $[V,V_0]$ is a linear combination of terms of the
left-most shape in this identity, hence the identity shows that $[V,V_0]
\subseteq V_1.$ Moreover, substituting for $s$ an element $[[x,c],[x,d]]
\in V_0$ we find that the last two terms are zero, since $[s,a] \in
V_1$ and $[x,[V_1,V_1]]=\{0\}.$ Hence the first term is also zero,
which shows that $[V_0,V_0]=\{0\}.$

Now let $V_2$ be the orthogonal complement $(V_0 \oplus V_1)^\perp,$ so
that $V$ decomposes orthogonally as $V_0 \oplus V_1 \oplus V_2.$ We claim
that $V_2$ is an ideal. First, we have $([V_0,V_2]|V)=(V_2|[V_0,V])
\subseteq (V_2|V_1)=\{0\},$ so $[V_0,V_2]=\{0\}.$ By the first
paragraph of the proof, $x$ is contained in $V_0,$ hence in particular
$[x,V_2]=0,$ so that $\ker L_x$ contains $V_0 \oplus V_2.$ For
dimension reasons, equality holds: $\ker L_x=V_0 \oplus V_2.$ Now
Lemma~\ref{lm:PreserveKernel} applied with $c=x$ yields that $\ker L_x$
is closed under multiplication, so in particular $[V_2,V_2]
\subseteq V_0 \oplus V_2.$ Since $([V_2,V_2]|V_0)=\{0\},$ we have
$[V_2,V_2] \subseteq V_2.$ Furthermore, we have
\[ ([V_1,V_2]|V_0 \oplus V_1)=(V_2|V_1 \oplus V_0)=\{0\}, \]
so that $[V_1,V_2] \subseteq V_2.$ This concludes the proof of the claim
that $V_2$ is an ideal. By simplicity of $V,$ $V_2=\{0\}$ and hence
$V=V_0 \oplus V_1.$

Now consider any $y \in V_0 \setminus \{0\}.$ Then $\ker L_y \supseteq V_0 \oplus V_2,$
and hence equality holds by maximality of $\dim \ker L_x.$ But we can show
more: let $v \in V_1$ be an eigenvector of the map $(L_x|_{V_1})^{-1}
(L_y|_{V_1})$ (which is linear since it is the composition of
two semilinear maps), say with eigenvalue $\lambda.$ Then we have
$[y,v]=[x,\lambda v]=[\lambda x,v].$  This means that the element
$z:=y-\lambda x \in V_0$ has $\ker L_z \supseteq V_0 \oplus V_2,$ but
also $v \in \ker L_x.$ Hence the kernel of $L_z$ is strictly larger
than that of $L_x,$ and therefore $z=0.$ We conclude that $y=\lambda x,$
and hence $V_0$ is one-dimensional.

Finally, consider a nonzero element $z \in V_1.$ From $[z,V_1]\subseteq
V_0=\langle x \rangle$ we find that $L_z V$ is contained in $\langle x,[z,x]
\rangle_\CC,$ i.e., $L_z$ has rank at most two. Hence, by minimality,
the same holds for $L_x.$ This means that $\dim V_1 \leq 2,$ and hence
$\dim V=\dim (V_0 \oplus V_1) \leq 3.$ Since $T$ is nonzero, we find
$\dim V=3,$ as desired.
\end{proof}

\section{Higher-order tensors} \label{sec:HigherOrder}

In this section, building on the case of order three, we prove the Main
Theorem for tensors of arbitrary order.

\subsection{Ordinary tensors}

Let $V_1,\ldots,V_d$ be finite dimensional inner product spaces over $K
\in \{\RR,\CC\}.$ The key observation is the following.  Let $J_1 \cup
\cdots \cup J_e=\{1,\ldots,d\}$ be a partition of $\{1,\ldots,d\}.$ Then
the natural {\em flattening} map
\[ V_1 \otimes \cdots \otimes V_d \to (\bigotimes_{j \in
J_1} V_j) \otimes \cdots \otimes (\bigotimes_{j \in J_e} V_j) \]
sends the set of order-$d$ odeco/udeco tensors into the set of
order-$e$ odeco/udeco
tensors, where the inner product on each factor $\bigotimes_{j \in J_\ell}
V_j$ is the one induced from the inner products on the factors. The
following proposition gives a strong converse to this observation.

\begin{prop} \label{prop:Flatten}
Let $T \in V_1 \otimes \cdots \otimes V_d$ be a tensor, where $d \geq
4.$ Suppose that the flattenings of $T$ with respect to the three
partitions
\begin{itemize}
\item[(i)] $\{1\},\ldots,\{d-3\},\{d-2\},\{d-1,d\},$
\item[(ii)] $\{1\},\ldots,\{d-3\},\{d-2,d-1\},\{d\},$ and
\item[(iii)] $\{1\},\ldots,\{d-3\},\{d-2,d\},\{d-1\}$
\end{itemize}
are all odeco/udeco. Then so is $T.$
\end{prop}

The lower bound of $4$ in this proposition is essential, because any
flattening of a three-tensor is a matrix and hence odeco, but as we have
seen in Section~\ref{sec:OrderThree} not every three-tensor is odeco.

\begin{proof}
As the first two flattenings are odeco, we have orthogonal decompositions
\[ T=\sum_{i=1}^k T_i \otimes u_i \otimes A_i=
\sum_{\ell=1}^r T'_\ell \otimes B_\ell \otimes w_\ell
\]
where $A_1,\ldots,A_k \in V_{d-1} \otimes V_{d}$ are pairwise orthogonal
and nonzero, and so are $u_1,\ldots,u_k \in V_{d-2}$, and the $T_i$
are of the form $z_{i1} \otimes \cdots \otimes z_{i(d-3)}$ where for
each $j$ the $z_{ij},\ i=1,\ldots$ are pairwise orthogonal and nonzero.
Similarly for the factors in the second expression. Contracting $T$
with $T_i$ in the first $d-3$ factors yields a single term on the left
(here we use that $d>3$):
\[ (T_i|T_i) u_i \otimes A_i = \sum_{\ell=1}^r (T'_\ell|T_i)
B_\ell \otimes w_\ell. \]
For an index $\ell$ such that $(T'_\ell|T_i)$ is nonzero, by contracting
with $w_\ell$ we find that $B_\ell$ is of rank one and, more specifically,
of the form $u_i \otimes v_\ell$ with $v_\ell \in V_{d-1}$. There is at
least one such index, since the left-hand side is nonzero.  Moreover,
since the $u_i$ are linearly independent for distinct $i$, we find
that the set of $\ell$ with $(T'_\ell|T_i) \neq 0$ is disjoint from the
set defined similarly for another value of $i$. Hence, $r \geq k$. By
swapping the roles of the two decompositions we also find the opposite
equality, so that $r=k$, and after relabelling we find $B_i=u_i \otimes
v_i$ for $i=1,\ldots,k$ and certain nonzero vectors $v_i$. Hence we find
\[ T=\sum_{i=1}^k T_i' \otimes u_i \otimes v_i \otimes w_i, \]
where we do not yet know whether the $v_i$ are pairwise perpendicular.
However, applying the same reasoning to the second and third
decompositions in the lemma, we obtain another decomposition
\[ T=\sum_{i=1}^k T_i' \otimes u_i' \otimes v_i' \otimes w_i, \]
where we do know that the $v_i'$ are pairwise perpendicular (but not
that the $u_i'$ are). Contracting with $T_i'$ we find that, in fact,
both decompositions are equal and the $v_i$ are pairwise perpendicular,
as required.
\end{proof}

\begin{proof}[Proof of the Main Theorem for ordinary tensors.]
It follows from Lemma~\ref{lm:OdecoAss} and Proposition~\ref{prop:AssOdeco},
that ordinary odeco tensors of order three are characterised by
degree-two equations. Similarly, by Lemma~\ref{lm:UdecoAss} and
Proposition~\ref{prop:AssUdeco}, ordinary udeco tensors of order three are
characterised by degree-three equations. By Proposition~\ref{prop:Flatten}
and the remarks preceding it, a higher-order tensor is odeco (udeco)
if and only if certain of its flattenings are odeco (udeco). Thus the
equations characterising lower-order odeco (udeco) tensors pull back,
along linear maps, to equations characterising higher-order odeco
(udeco) tensors.
\end{proof}

\subsection{Symmetric tensors}

In this section, $V$ is a finite-dimension vector space over $K=\RR$
or $\CC.$

\begin{prop}\label{prop:SymAsOrdinary}
For $d  \geq 3,$ a tensor $T \in \Sym_d(V)$ is symmetrically odeco (udeco)
if and only if it is odeco (udeco) when considered as an ordinary tensor
in $V^{\otimes d}.$
\end{prop}

\begin{proof}
The ``only
if'' direction is immediate, since a symmetric orthogonal decomposition is
{\em a fortiori} an ordinary orthogonal decomposition. For the converse,
consider an orthogonal decomposition
\[ T=\sum_{i=1}^k v_{i1} \otimes \cdots \otimes v_{id}, \]
where the $v_{ij}$ are nonzero vectors, pairwise perpendicular for fixed
$j.$ Since $T$ is symmetric, we have
\[ T=\sum_i v_{i\pi(1)} \otimes \cdots \otimes v_{i\pi(d)} \]
for each $\pi \in S_d.$ By uniqueness of the decomposition
(Proposition~\ref{prop:Unique}), the terms in this latter decomposition
are the same, up to a permutation, as the terms in the original
decomposition. In particular, the unordered cardinality-$k$ sets
of projective points $Q_j:=\{[v_{1j}],\ldots,[v_{kj}]\} \subseteq
\PP V$ are identical for all $j=1,\ldots,d.$ Consider the integer $k
\times d$-matrix $A$ with entries in $[k]:=\{1,\ldots,k\}$ determined by
$a_{ij}=m$ if $[v_{ij}]=[v_{m1}].$ This matrix has all integers $1,\ldots,k$
in each column, in increasing order in the first column, and
furthermore has the property that for each $d \times d$-permutation
matrix $\pi$ there exists a $k \times k$-permutation matrix $\sigma$ such
that $\sigma A=A \pi.$ To conclude the proof we only need to prove the following claim, namely
that, for $d \geq 3,$ the only such $(k \times d)$-matrix is the
matrix whose $i$-th row consists entirely of copies of $i.$

\emph{Claim.}
Let $k\geq 1$ and $d\geq 3$ be natural numbers. Let $S_k$ act on $S_k^d$
diagonally from the left by left multiplication and let $S_d$ act on
$S_k^d$ from the right by permuting the terms. Consider an element
\[A:=(\mathrm{id}, \tau_2, \ldots, \tau_{d}) \in S_k^d,\]
where $\mathrm{id}$ is the identity permutation.
Suppose that for each $\pi \in S_d$ there exists a
$\sigma \in S_k$ such
that $\sigma A=A \pi.$ Then $A=(\mathrm{id},\ldots, \mathrm{id}).$

\emph{Proof of claim.} For $j \in \{2,\ldots,d\}$ pick $\pi_j=(1,j)$
to be the transposition switching $1$ and $j$. By the property imposed
on $A$ there exists a $\sigma_j$ such that $\sigma_j A= A \pi_j.$
In particular, $(A\pi_j)_1=\tau_j$ equals $(\sigma_j A)_1=\sigma_j.$
So $\tau_j=\sigma_j$ for all $j \in \{2,\ldots,d\}$. Since $d\geq 3,$
one can pick an index $l$ which is fixed by $\pi_j,$ so that
$\tau_l=(\sigma_j A)_l=\sigma_j \tau_l.$ So then
$\sigma_j=\mathrm{id}=\tau_j.$ This concludes the proof of the claim, and thus that of Proposition ~\ref{prop:SymAsOrdinary}.
\end{proof}

\begin{proof}[Proof of the Main Theorem for symmetric tensors.]
By the preceding proposition, the equations for odeco tensors in
$V \otimes \cdots \otimes V$ pull back to equations characterising
symmetrically odeco tensors in $\Sym_d V$ via the inclusion of the
latter space into the former. Thus the Main Theorem for symmetric
tensors follows from the Main Theorem for ordinary tensors, proved in
the previous subsection.
\end{proof}

\begin{re}
The proof of the Main Theorem in Section~\ref{sec:OrderThree} for
ordinary odeco three-tensors relies on the proof for symmetrically
odeco three-tensors, so the proof above does not render that proof
superfluous. On the other hand, the proof for ordinary {\em udeco}
three-tensors does not rely on that for symmetrically udeco three-tensors,
so in view of the proof above the latter could have been left out.
We have decided to retain it for completenes.
\end{re}

\begin{re} \label{re:AlternatingNotOdeco}
The argument in the proposition also implies that an 
odeco/udeco
tensor in $V^{\otimes d} \setminus \{0\}$ with $d \geq 3$ cannot be alternating: permuting
tensor factors with a transposition must leave the decomposition intact
up to a sign and a permutation of terms, but then the claim shows that in
each term all vectors are equal, hence their alternating product is zero.
\end{re}

\subsection{Alternating tensors}\label{subsec:HigherOrderAlt}

In this section we prove that an alternating tensor of order at least four
is alternatingly odeco/udeco if and only if all its contractions with
a vector are. Thus, let $V$ be a vector space over $K \in \{\RR,\CC\}$
and consider an orthogonal decomposition
\begin{equation}  \label{eq:altT}
T=\sum_{i=1}^k \lambda_i v_{i1} \wedge \cdots \wedge v_{id}
\end{equation}
of an alternatingly odeco tensor $T \in \Alt_d V,$ where
$v_{11},\ldots,v_{kd}$ form an orthonormal set of vectors in $V$ and where
$\lambda_i \in K.$ The following lemmas are straightforward exercises in
differential geometry, and we omit their proofs.

\begin{lm} \label{lm:Tangent}
Suppose that $K=\RR.$ Let $d \geq 3$ and $dk \leq n:=\dim V.$ The set
$X$ of alternatingly odeco tensors in $\Alt_d V$ with exactly $k$ terms
in their orthogonal decomposition is a smooth manifold of dimension
$k+\frac{1}{2}dk(2n-(k+1)d)$ whose tangent space at a point $T$ is the direct
sum of the following spaces:
\begin{enumerate}
\item $\bigoplus_{i=1}^k (\Alt_{d-1} V_i) \wedge V_0$ where
$V_i=\langle v_{i1},\ldots,v_{id} \rangle$ and $V_0=(V_1
\oplus \cdots \oplus V_k)^\perp;$
\item $\bigoplus_{i=1}^k \Alt_d V_i;$ and
\item $\langle \lambda_i (v_{i1} \wedge \cdots \wedge
v_{ml} \wedge \cdots \wedge
v_{id}) - \lambda_m (v_{m1} \wedge \cdots \wedge
v_{ij} \wedge \cdots \wedge v_{md})
\rangle \mid 1 \leq j,l \leq d \text{ and } i \neq m
\rangle,$ where $v_{ml}$ replaces $v_{ij}$ in the first
term and vice versa in the second term.
\end{enumerate}
\end{lm}

The three summands are obtained as follows: $X$ is
the image of the Cartesian product of the manifold of $k \cdot d$-tuples
of orthonormal vectors with $(\RR \setminus \{0\})^k$ via
\[ \phi:((v_{ij})_{(i,j) \in [k] \times [d]}, \lambda) \mapsto
\sum_i \lambda_i v_{i1} \wedge \cdots \wedge v_{id}. \]
Replacing a $v_{ij}$ by a $v_{ij}+\epsilon v_0$ with $v_0
\in V_0$ yields the first summand. Replacing $\lambda_i$ by
$\lambda_i+\epsilon$ yields the second summand, and
infinitesimally rotating $(v_{ij},v_{ml})$ into
$(v_{ij}+\epsilon v_{ml}, v_{ml}-\epsilon v_{ij})$ yields
the last summand. The complex analogue of Lemma \ref{lm:Tangent} is the following.

\begin{lm} \label{lm:TangentC}
Suppose that $K=\CC.$ Let $d \geq 3$ and $2k \leq n:=\dim_\CC V.$ The
set $X$ of alternatingly udeco tensors in $\Alt_d V$ with exactly $k$
terms in their orthogonal decomposition is a smooth manifold of dimension
$2k+dk(2n-(k+1)d)$ whose tangent space at $T$ is the direct sum of the
following spaces:
\begin{enumerate}
\item the complex space $\bigoplus_{i=1}^k (\Alt_{d-1} V_i) \wedge V_0$
where $V_i=\langle v_{i1},\ldots,v_{id} \rangle$ and $V_0=(V_1 \oplus
\cdots \oplus V_k)^\perp;$
\item the complex space $\bigoplus_{i=1}^k \Alt_d V_i;$
\item the real space $\langle \lambda_i (v_{i1} \wedge \cdots \wedge
v_{ml} \wedge \cdots \wedge
v_{id}) - \lambda_m (v_{m1} \wedge \cdots \wedge
v_{ij} \wedge \cdots \wedge v_{md})
\rangle \mid 1 \leq j,l \leq d \text{ and } i \neq m
\rangle_\RR,$ where $v_{ml}$ replaces $v_{ij}$ in the first
term and vice versa in the second term; and
\item the real space $\langle \lambda_i (v_{i1} \wedge \cdots \wedge
(\bi v_{ml}) \wedge \cdots \wedge
v_{id}) + \lambda_m (v_{m1} \wedge \cdots \wedge
(\bi v_{ij}) \wedge \cdots \wedge v_{md})
\rangle \mid 1 \leq j,l \leq d \text{ and } i \neq m
\rangle_\RR,$ where $\bi v_{ml}$ replaces $v_{ij}$ in the first
term and vice versa in the second term and where $\bi \in \CC$ is a
square root of $-1.$
\end{enumerate}
\end{lm}

The last summand arises from the infinitesimal unitary transformations
sending $(u_{ij},u_{ml})$ to $(u_{ij}+\bi u_{ml},u_{ml}+\bi u_{ij}).$

\begin{prop} \label{prop:AltContraction}
Let $V$ be a vector space over $K \in \{\RR,\CC\}.$  Let $d \geq 3$ and
let $S \in \Alt_{d+1} V.$ Then $S$ is alternatingly odeco (or udeco) if and
only if for each $v_0 \in V$ the contraction $(S|v_0)$ of $S$ with $v_0$
in the last factor is an alternatingly odeco (or udeco) tensor in $\Alt_d V.$
\end{prop}

\begin{proof}
The ``only if'' direction is immediate: contracting the terms in
an orthogonal decomposition of $S$ with $v_0$ yields an orthogonal
decomposition for $(S|v_0).$ Note that in this process the pairwise
orthogonal $(d+1)$-spaces encoded by $S$ are replaced by their
$d$-dimensional intersections with the hyperplane $v_0^\perp,$ and
discarded if they happen to be contained in that hyperplane.

Conversely, assume that all contractions of $S$ with a vector are
alternatingly odeco.  Among all $v_0 \in V$ choose one, say of norm $1,$
such that $T:=(S|v_0)$ is odeco with the maximal number of terms, say $k,$
and let $\lambda_i$ and the $v_{ij}$ be as in \eqref{eq:altT}. Then $\Psi:
v \mapsto (S|v)$ is a real-linear map from an open neighbourhood of $v_0$
in $V$ into the set $X$ in the lemma, and hence its derivative at $v_0,$
which is the $\Psi$ itself, maps $V$ into the tangent space described in
the lemma. Since contracting with $v_0$ maps $\Alt_{d+1} V$ into $\Alt_d
(v_0^\perp),$ we may choose a basis $v_{00},\ldots,v_{0(n-kd)}$ of $V_0$
from the lemma that starts with $v_{00}:=v_0.$ Now we have
\[ S=\left(\sum_{i=1}^k \lambda_i v_{i1} \wedge \cdots \wedge v_{id}
\wedge v_{00}\right) + S''=:S'+S'' \]
where $(S''|v_{00})=0.$ We have an orthonormal basis $(v_{ij})_{ij}$
of $V$ where $(i,j)$ runs through $A:=([k]\times [d]) \cup (\{0\} \times
[n-kd]),$ where $[k]:=\{1,\ldots,k\}.$ 

For a subset $I \subseteq A$ we write $v_I$ for the vector in
$\Alt_{d+1} V$ obtained as the wedge product of the vectors labelled by
$I$ (in some fixed linear order on $A$). The vectors $v_I$ with $|I|=d+1$
form a $K$-basis of $\Alt_{d+1} V,$ and similarly for those with $|I|=d.$ Now
$(S'|v)$ lies in the tangent space to $X$ at $T$ for all $v$ (indeed,
in the sum of the first two summands in the lemma). Hence also $(S''|v)$
must lie in that tangent space. Expand $S''$ on the chosen basis:
\[ S''=\sum_{I \subseteq A, |I|=d+1} c_I v_I. \]
We claim that $c_I=0$ unless $I$ contains one of the $k$ sets $\{i\}
\times [d].$ Indeed, suppose that $c_I \neq 0$ and that $I$ does not
contain any of these $k$ sets. Contracting $v_I$ with any $v_\alpha$ with
$\alpha \in I$ yields $\pm v_J$ where $J:=I \setminus \{\alpha\}$,
hence $v_J$ appears with a nonzero coefficient
in $(S''|v_\alpha).$  By the lemma we find that $J$
must contain a $(d-1)$-subset of at least one of the sets $\{i\} \times
[d].$ So in particular, there exists an $i$ such that $I$ itself contains
a $(d-1)$-subset of $\{i\} \times [d].$  Suppose first that this $i$
is unique, say equal to $i_0.$ Then contracting $v_I$ with $v_{i_0,j}$
with $(i_0,j) \in I$ yields $\pm v_J$ where $J$ contains only at most $d-2$
of the elements of each of the sets $\{i\} \times [d],$ a contradiction
with the lemma. So this $i$ is not unique. Then $I$ contains $d-1$
elements from each of at least two disjoint sets, so $2(d-1) \leq d+1,$
so $d\leq 3,$ and hence $d=3$---here we use that $d \geq 3.$ Without loss
of generality, then, $I=\{(1,1),(1,2),(2,1),(2,2)\}.$ Now contracting
$v_I$ with $v_{11}$ yields a scalar times $\pm v_{12} \wedge v_{21}
\wedge v_{22},$ hence this term appears in $(S|v_{11}).$ But (see the
last one/two summand/summands in the tangent space for the odeco/udeco
case, respectively) this term can only appear in a tangent vector if
also the term $\pm v_{11} \wedge v_{23} \wedge v_{13}$ appears---which
is impossible after contracting with $v_{11}.$ This proves the claim.

We conclude that $S$ can be written as
\[ S=\sum_{i=1}^k v_{i1} \wedge \cdots \wedge v_{id} \wedge
w_i \]
for suitable vectors $w_i$ satisfying $(w_i|v_{0})=\lambda_i.$ Set
$W_i:=V_i + \langle w_i \rangle.$ We need to show that the spaces
$W_1,\ldots,W_k$ are pairwise perpendicular. For this, it suffices to show
that, for $z$ in an open dense subset of $V,$ the spaces $W_i':=W_i \cap
z^\perp$ are pairwise perpendicular. We choose this open subset such that
\begin{enumerate}
\item the contraction $(S|z)$ has an orthogonal decomposition with $k$
terms;
\item the $k$ spaces $W_i'$ are $d$-dimensional and linearly independent;
\item the tensor $((S|z)|v_{0})=\pm((S|v_0)|z) \in \Alt_{d-1}V,$ which by assumption
is alternatingly odeco, has a unique orthogonal decomposition.
\end{enumerate}
By proposition~\ref{prop:Unique}, the last condition is void if $d>3.$
Now, each $W_i'':=W_i' \cap v_0^\perp$ is contained in $V_i,$ so that $W_i'' \perp
W_m''$ for all $i \neq m.$ Now, by assumption, the tensor
\[ (S|z) \in \bigoplus_{i=1}^k \Alt_d W_i' \]
is alternatingly odeco with $k$ terms. Let $U_1,\ldots,U_k$ be
the $d$-dimensional, pairwise orthogonal spaces encoded by it. Then
$((S|z)|v_{0})$ has an orthogonal decomposition with terms in $\Alt_{d-1}
(U_i \cap v_{0}^\perp).$ But we also have
\[ ((S|z)|v_{0}) \in \bigoplus_{i=1}^k \Alt_{d-1} W_i'', \]
where the $W_i''$ are pairwise perpendicular. So, since we assumed
that this orthogonal decomposition is unique, after a permutation
of the $U_i$ we have $U_i \cap v_{0}^\perp=W_i''.$ Now let
$u_{i1},\ldots,u_{id}$ be an orthonormal basis of $U_i,$ where the first
$(d-1)$ form a basis of $W_i''.$ Extend with $u_{01},\ldots,u_{0(n-kd)}$
to an orthonormal basis of $V.$ Arguing with respect to the basis
$(u_I)_{|I|=d},$ we find that the map $V^k \to
\Alt_d V$ that sends $(y_1,\ldots,y_k)$ to $\sum_{i=1}^k u_{i1} \wedge
\cdots \wedge u_{i(d-1)} \wedge y_i$ is injective. Since
\[ (S|z)=\sum_{i=1}^k \mu_i u_{i1} \wedge \cdots \wedge u_{id}
=\sum_{i=1}^k \mu'_i u_{i1} \wedge \cdots \wedge u_{i(d-1)} \wedge w'_i
\]
for suitable $w'_i \in W'_i$ and nonzero scalars $\mu_i,\mu_i',$ we find that $W_i'=U_i,$ and hence the
$W_i'$ are pairwise perpendicular, as desired.
\end{proof}

\begin{proof}[Proof of the Main Theorem for alternating tensors.]
In Lemmas~\ref{lm:AltOdecoJacobi}, \ref{lm:AltOdecoCasimir}
and Pro\-position~\ref{prop:JacobiCasimirAltOdeco} we found
that an alternating
three-tensor is alternatingly odeco if and only if it satisfies
certain polynomial equations of degrees $2$ and $4.$
Correspon\-ding\-ly,
Proposition~\ref{prop:IdsCasimirAltUdeco}
settles the Main Theorem for alternatingly {\em udeco}
three-tensors. Proposition~\ref{prop:AltContraction} yields that the
pullbacks of the real polynomial equations characterising alternatingly
odeco/udeco $d$-tensors along real-linear maps yield equations
characterising alternatingly odeco/udeco $(d+1)$-tensors. These
pullbacks have the same degrees as the original equations.
\end{proof}

\section{Concluding remarks} \label{sec:Concluding}

We have established low-degree real-algebraic characterisations of
orthogonally decomposable tensors in six different scenarios. While this
is quite a satisfactory result, at least three questions remain.

First, do the equations that we have found generate the ideals of
the real-algebraic varieties at hand? We are somewhat optimistic
in the ordinary and symmetric odeco case, because of evidence in
\cite{robeva_odeco} for the case of symmetrically odeco $2 \times 2 \times
\cdots \times 2$-tensors. But in general we believe that representation
theory of the orthogonal and unitary groups should be used to approach
this question.

Second, and related to this, our main result can be read as a finiteness
result for an infinite class of varieties in the spirit of Snowden's
Delta-modules \cite{Snowden10}. Can the methods of \cite{Sam15},
tailored to the orthogonal and unitary groups that preserve orthogonally
decomposable tensors, lead to more refined finiteness results on equations
and higher-order syzygies?

Third, a potentially interesting line of research, which we have not yet
pursued further, concerns a form of (non-associative, non-commutative)
{\em elimination}. To make this somewhat precise, suppose that we are
given a number of polynomial identities defining a class of algebras
over $\RR.$ Now consider the functor that assigns to such an algebra
$A$ the space $\CC \otimes_\RR A$ equipped with the {\em semilinear}
extension of the product, and that assigns to an algebra homomorphism the
its {\em linear} extension. What polynomial identities are satisfied by
the image of our class under this functor? Above we have implicitly seen
that commutative, associative $\RR$-algebras are mapped to commutative,
semi-associative $\CC$-algebras. But in the case of real Lie algebras
we do not know a characterisation of the outcome---this is why we needed
more {\em ad hoc} methods for alternatingly udeco three-tensors.

\section*{Acknowledgments} 
AB and JD are supported by a Vidi grant from the from the Netherlands
Organisation for Scientific Research (NWO). EH is supported by an NWO
free competition grant. AB partially supported by MIUR funds, PRIN 2010--2011
project ``Geometria delle variet\`a algebriche'' and
Universit\`a degli Studi di Trieste--FRA 2011 project
``Geometria e topologia delle variet\`a''. AB is member of INdAM-GNSAGA. ER is supported by the UC Berkeley Mathematics Department. We thank Nick Vannieuwenhoven for several
remarks on a previous draft. Finally, we thank the organizers of the Fall 2014 workshop ``Tensors in Computer Science and Geometry'' at the Simons Institute for the Theory of Computing, where this project started in embryo. 

\bibliography{odeco}{}
\bibliographystyle{amsalpha}

\end{document}